\documentclass[11pt]{amsart}
\usepackage[mathscr]{eucal}
\usepackage{amssymb, amsmath,array, amscd, leftidx}
\usepackage{enumerate,verbatim}
\usepackage{graphicx}
\usepackage{url}
\usepackage[colorlinks,plainpages]{hyperref}
\usepackage[width=5.65in,height=7.8in,centering]{geometry}
\usepackage[all]{xy}

\newtheorem{theorem}{Theorem}[section]

\newtheorem{corollary}[theorem]{Corollary}
\newtheorem{lemma}[theorem]{Lemma}
\newtheorem{proposition}[theorem]{Proposition}

\theoremstyle{definition}
\newtheorem{definition}[theorem]{Definition}
\newtheorem{example}[theorem]{Example}
\newtheorem{remark}[theorem]{Remark}

\newtheorem*{ack}{Acknowledgment}

\numberwithin{equation}{section}

\newcommand{\abs}[1]{\left|#1\right|}
\newcommand{\set}[1]{\left\{#1\right\}}
\newcommand{\ip}[1]{\left<#1\right>}
\renewcommand{\b}[1]{\mathbf{#1}}   

\newcommand{\dd}[2]{\frac{\partial #1}{\partial #2}}

\newcommand{\A}{\mathcal{A}}
\newcommand{\der}{{\rm Der}}

\newcommand{\derA}{\der(\mathcal{A})}
\newcommand{\OSigma}{{\overline{\Sigma}}}  
\newcommand{\omsheaf}{{\widetilde{\Omega}}}
\newcommand{\C}{\mathbb{C}}
\newcommand{\Z}{\mathbb{Z}}
\renewcommand{\P}{{\mathbb P}}
\renewcommand{\d}{{\rm d}}
\newcommand{\la}{{\lambda }}
\newcommand{\LL}{{\mathcal L}}
\newcommand{\SL}{{\mathfrak{sl}_2}}
\newcommand{\Imer}{{I_{\rm mer}}}  
\newcommand{\Ilog}{I}              
\newcommand{\IQ}{{I'}}             
\newcommand{\m}{{\mathfrak m}}     
\newcommand{\p}{{\mathfrak p}}     
\newcommand{\kk}{{C}}              

\def\dot{\mathchar"013A}  
\newcommand{\hdot}{{\hskip-0.05em\raise1pt\hbox to0.35em{\huge $\dot$}}}

\DeclareMathOperator{\tdeg}{tdeg}
\DeclareMathOperator{\tor}{Tor}
\DeclareMathOperator{\Ext}{Ext}
\DeclareMathOperator{\Hom}{Hom}
\DeclareMathOperator{\rad}{rad}

\begin{document}

\title{Critical points and resonance of hyperplane arrangements}

\author[D. Cohen]{D. Cohen$^1$}
\address{Department of Mathematics, Louisiana State University,
Baton Rouge, LA 70803, USA}
\email{\href{mailto:cohen@math.lsu.edu}{cohen@math.lsu.edu}}
\urladdr{\href{http://www.math.lsu.edu/~cohen/}
{www.math.lsu.edu/\char'176cohen}}
\thanks{{$^1$}Partially supported 
by National Security Agency grant H98230-05-1-0055}

\author[G. Denham]{G. Denham$^2$}
\address{Department of Mathematics, University of Western Ontario\\
London, ON  N6A 5B7, Canada}
\urladdr{\href{http://www.math.uwo.ca/~gdenham}%
{www.math.uwo.ca/\char'176gdenham}}
\thanks{{$^2$}Partially supported by a grant from NSERC of Canada}

\author[M. Falk]{M. Falk}
\address{Department of Mathematics and Statistics, 
Northern Arizona University, Flagstaff, AZ 86011, USA}
\email{\href{mailto:michael.falk@nau.edu}{michael.falk@nau.edu}}
\urladdr{\href{http://www.cefns.nau.edu/~falk/}%
{www.cefns.nau.edu/\char'176falk}}

\author[A. Varchenko]{A. Varchenko$^4$}
\address{Department of Mathematics, University of
North Carolina at Chapel Hill, Chapel Hill, NC 27599, USA}
\email{\href{mailto:anv@email.unc.edu}{anv@email.unc.edu}}
\urladdr{\href{http://www.math.unc.edu/Faculty/av/}
{www.math.unc.edu/Faculty/av/}}
\thanks{{$^4$}Partially supported by NSF grant DMS-0555327}

\subjclass[2000]{Primary
32S22,  
Secondary
55N25,  
52C35.  
}

\keywords{hyperplane arrangement, master function, resonant weights, 
critical set}

\begin{abstract}
If $\Phi_\la$ is a master function corresponding to a hyperplane arrangement 
$\A$ and a collection of weights $\la$, we investigate the relationship 
between the critical set of $\Phi_\la$, the variety defined by the vanishing 
of the one-form $\omega_\la=\d \log \Phi_\la$, and the resonance of $\la$.  
For arrangements satisfying certain conditions, we show that if $\la$ is 
resonant in dimension $p$, then the critical set of $\Phi_\la$ has codimension 
at most $p$.  These include all free arrangements and all rank $3$ arrangements.
\end{abstract}

\maketitle

\date{\today}

\section{Introduction} \label{sec:intro}
Let $\A=\set{H_1,\dots,H_n}$ be an arrangement of hyperplanes in
$V=\C^\ell$, 
with complement $M=M(\A)=V\setminus\bigcup_{j=1}^n H_j$.  
Fix coordinates $\b{x}=(x_1,
\dots,x_\ell)$ on $V$, and for each hyperplane $H_j$ 
of $\A$, let $f_j$ be a linear polynomial for which $H_j=
\{\b{x}\mid f_j(\b{x})=0\}$. 
A collection $\lambda=
(\la_1,\dots,\la_n)\in \C^n$ of complex weights determines a
\emph{master function}
\begin{equation} \label{eq:master}
\Phi_\la = \prod_{j=1}^n f_j^{\la_j},
\end{equation}
a multi-valued holomorphic function with zeros and poles on the variety 
$\bigcup_{j=1}^n H_j$ defined by $\A$.  The master function $\Phi_\la$ 
determines a one-form 
\begin{equation} \label{eq:omega}
\omega_\la = \d \log \Phi_\la = \sum_{j=1}^n \la_j \frac{\d f_j}{f_j}
\end{equation}
in the Orlik-Solomon algebra $A(\A) \cong H^\hdot(M;\C)$, a quotient of an 
exterior algebra. 

Two focal points in the recent study of arrangements are the cohomology 
$H^\hdot(A(\A),\omega_\la)$ of the Orlik-Solomon algebra with differential 
given by multiplication by $\omega_\la$, and the critical set of the 
master function $\Phi_\la$, the variety $V(\omega_\la) \subset M$
defined by the vanishing of the one-form $\omega_\la$.  We shall
denote the latter by $\Sigma_\lambda$.  The cohomology 
$H^\hdot(A(\A),\omega_\la)$ arises in the study of local systems on $M$.  
Under certain conditions on the weights $\la$, the inclusion of 
$(A(\A),\omega_\la)$ in the twisted de Rham complex 
$(\Omega^\hdot(*\A), \d+\omega_\la)$ induces an isomorphism 
$H^\hdot(A(\A),\omega_\la) \cong H^\hdot(M;\LL_\la)$, where $\LL_\la$ is the 
complex, rank one local system on $M$ with monodromy 
$\exp(-2\pi \sqrt{-1}\,\la_j)$ about the hyperplane $H_j$.  
See \cite{OTmsj} for discussion of these results and applications to 
hypergeometric integrals.  The critical set of the master function is 
also of interest in mathematical physics.  For instance, for certain 
arrangements, the critical equations of the $\Phi_\la$ coincide with 
the Bethe ansatz equations for the Gaudin model associated with a 
complex simple Lie algebra $\mathfrak g$, see~\cite{RV,Va06}.

Assume that $\A$ contains $\ell$ linearly independent hyperplanes, and 
note that $M$ has the homotopy type of an $\ell$-dimensional cell
complex.  
For generic weights $\la$, the cohomology groups $H^q(A(\A),\omega_\la)$
vanish in all dimensions except possibly $q=\ell$, and 
$\dim H^\ell(A(\A),\omega_\la) = \abs{\chi(M)}$, where $\chi(M)$ is the 
Euler characteristic of $M$, see Yuzvinsky \cite{Yuz}.  
Those weights $\la$ for which the cohomology
does not vanish (in dimension $q\neq \ell$) are said to be resonant, and
comprise the resonance varieties
\[
R^{q}_p(A(\A))=\{\la\in\C^{n} \mid 
\dim H^{q}(A(\A),\omega_{\la})\ge p\},\quad 0<q<\ell,\ 0<p.  
\] 

In \cite{Varchenko}, Varchenko conjectured that, for generic weights
$\la$, 
the master function $\Phi_\la$ has $|\chi(M)|$ nondegenerate critical
points 
in $M$, and proved this result in the case where the hyperplanes of $\A$
are defined by real linear polynomials $f_j$.  Varchenko's conjecture 
was established for an arbitrary arrangement $\A$ by Orlik and Terao 
\cite{ot95}.  See Damon~\cite{Damon} and Silvotti \cite{Silvotti} for 
generalizations.
For generic, or nonresonant, weights $\la$, the critical set of
$\Phi_\la$ 
was used to construct a basis for the local system homology group 
$H_\ell(M;\LL_\la)$ by Orlik and Silvotti \cite{OrlikSilvotti}.

Let $z=(z_1,\dots,z_n)$ be an $n$-tuple of distinct complex numbers,
$z_i \neq z_j$ for $i \neq j$, $m=(m_1,\dots,m_n)$ an
$n$-tuple of nonnegative integers, and $\kappa \in \C^*$ generic.  The 
master function
\begin{equation} \label{eqn:disc master}
\Phi_{\ell,n}= \prod_{i=1}^\ell \prod_{j=1}^n 
(x_i - z_j)^{-m_j/\kappa} \prod_{1\le p < q \le \ell}(x_p -
x_q)^{2/\kappa}
\end{equation}
defines a local system on the complement of the 
Schechtman-Varchenko discriminantal arrangement $\A_{\ell,n}$
corresponding 
to the $\SL$ KZ differential equations, see~\cite{SV}.  The critical set
of $\Phi_{\ell,n}$ 
was determined by Scherbak and Varchenko \cite{ScherbakVarchenko}.  
Let $\abs{m}=\sum_{j=1}^n m_j$.  If $m$ satisfies $0\le
|m|-\ell+1<\ell$, 
then for generic $z$, the critical set of $\Phi_{\ell,n}$ consists of a 
certain number, say $k$, of curves in $V$, see 
\cite[Thm. 1]{ScherbakVarchenko}. Let $\la=(\ldots, -m_j/\kappa,
\ldots,2/\kappa,\ldots)$ denote the associated collection of weights.  
The Orlik-Solomon cohomology $H^\hdot(A(\A_{\ell,n}), \omega_\la)$ was 
subsequently studied by Cohen and Varchenko \cite{CV}.  Under the same 
conditions on $m$, this cohomology is nontrivial in codimension one, 
$H^{\ell-1}(A(\A_{\ell,n}), \omega_\la) \neq 0$.  Furthermore, the 
dimension of the subspace of skew-symmetric cohomology classes 
under the natural action of the symmetric group $S_\ell$ is equal to
$k$, 
the number of components of the critical set, see \cite[Thm. 1.1]{CV}.
Mukhin and Varchenko~\cite{MV,MV2} also describe interesting multidimensional
critical sets for master functions generalizing 
those of \eqref{eqn:disc master} to other root systems.

These results suggest a relationship between the critical set $\Sigma_\lambda$
and the resonance, or nonvanishing, of $H^\hdot(A(\A),\omega_\la)$.  The 
main purpose of this note is to establish such a relationship for
tame arrangements, defined below.
Our main result, Theorem~\ref{thm:res2crit},
insures that if $H^p(A(\A),\omega_\la) \neq 0$, then the
codimension of the critical set of the master function $\Phi_\la$ is at 
most $p$, as long as one of the following conditions holds: $\A$ is
free; $\A$ has rank
$3$; $\A$ is tame and $p\leq 2$.   

Some of  the results presented here were announced in \cite{Den07} and \cite{Fal07}. 
These reports inspired Dimca \cite{Dim08} to find other conditions 
which insure that the codimension of $Z(\omega_\la)$ is at most $p$, where $Z(\omega_\la)$ 
is the zero set of $\omega_\la$ in a good compactification of $M$ and it is additionally assumed that 
$H^j(A(\A),\omega_\la)=0$ for $j<p$.

Our main result is proven in two steps.  First, in \S\ref{sec:geom},
we develop some properties of a variety $\Sigma(\A)\subseteq M\times
\C^n$ that parameterizes all critical sets for a fixed arrangement $\A$.  
Its closure in affine space $\C^\ell\times\C^n$, denoted by $\OSigma(\A)$,
can be described in terms of logarithmic derivations; we show that
the variety is a complete intersection if and only if $\A$ is free.
We also find that $\Sigma(\A)$ is arithmetically 
Cohen-Macaulay if $\A$ is tame, although
we do not know if the converse holds or not.

Let $R$ be the coordinate ring of $V=\C^\ell$, and identify $R$ with the 
polynomial ring $\C[x_1,\dots,x_\ell]$.  
Assume that $\A$ is a central arrangement, so 
that each hyperplane of $\A$ passes through the origin in $V$.  
We will see (\S\ref{subsec:affine}) that this assumption causes no
loss of generality.
The polynomials $f_j$ defining the hyperplanes of $\A$ are then
linear forms, and a defining polynomial $Q = \prod_{j=1}^n f_j$ of $\A$ 
is homogeneous of degree $n=\abs{\A}$.  For any $k$-algebra $T$, 
let $\der_k(T)$ denote the
$T$-module of $k$-linear derivations on $T$.  Let $\derA$ denote the
module of logarithmic derivations on $M(\A)$:
\begin{equation}\label{eq:defDer}
\derA=\set{\theta\in\der_\C(R)\colon\theta(Q)\in(Q)}.
\end{equation}
The arrangement $\A$ is said to be free if the module $\derA$
is a free $R$-module.  

The notion of a tame arrangement first arose in \cite{ot95b} and subsequently
appeared in \cite{YT95,WiensYuz}.  Tame arrangements include generic 
arrangements, free arrangements (hence discriminental arrangements), and all 
arrangements of dimension less than $4$.  The precise definition appears in 
the next section: see Definition~\ref{def:tame}.

In Section \S\ref{sec:koszul}, we use a complex of logarithmic differential
forms to resolve the defining ideal of $\OSigma(\A)$.  For free arrangements,
this is simply a Koszul complex.  The general case is more awkward, since
the resolution is not in general free, and we require the tame hypothesis
to show that it is exact.  Nevertheless, this provides a link relating
the codimension of a critical set and nonvanishing of the cohomology 
$H^\hdot(\Omega^\hdot(\A),\omega_\lambda)$ of the complex of 
logarithmic forms with poles along $\A$ 
(Theorem~\ref{th:two} and corollaries).

The second step is to show that $H^p(\Omega^\hdot(\A),\omega_\lambda)\neq0$
implies that $H^p(A(\A),\omega_\lambda)\neq0$, which we do in 
\S\ref{sec:res2crit}.
The argument combines a result of Wiens and Yuzvinsky~\cite{WiensYuz}
(which requires the ``tame'' hypothesis again) 
with a spectral sequence due to Farber \cite{farber}.  In \S\ref{sec:examples}, we
give some examples which show, in particular, that the reverse implication
does not hold in general.

\section{Geometry of the critical set}\label{sec:geom}
In this section, we introduce and compare several slightly different 
algebraic descriptions of critical sets of master functions.
In particular, we recall that for each arrangement $\A$ of $n$ hyperplanes, 
there
exists a manifold of dimension $n$ that parameterizes the critical sets
of all master functions on $\A$.  

\subsection{Central and irreducible arrangements}\label{subsec:affine}
We will want to make two reductions to the class of arrangements considered
in the arguments that follow.  First, it is sufficient to consider
arrangements which are central.  For this, 
if $\A=\set{H_j}_{j=1}^n$ is a non-central arrangement in $\C^{\ell-1}$
with master function $\Phi_\lambda$, we homogenize
the equations $\set{f_j}$ by adding a new variable $x_0$, and introduce a new
hyperplane $H_0$ defined by $f_0=x_0$ with weight $\lambda_0=-\sum_{i=1}^n\lambda_i$. 
This yields a central arrangement $\A'$ in $\C^\ell$ (the cone of $\A$), with weights 
$\la'=(\la_0,\la_1,\dots,\la_n)$, and corresponding master function $\Phi_{\la'}$.  
If $\Sigma_{\la'}$ is the critical set of $\Phi_{\la'}$, then $\Sigma_\lambda$ can be identified with $\P\Sigma_{\lambda'}$ by
restricting to the affine chart of $\P^{\ell-1}$ with $x_0\neq0$.
Accordingly, the codimensions of $\Sigma_\lambda$ in $\C^{\ell-1}$,
of $\Sigma_{\lambda'}$ in $\C^\ell$, and of $\P\Sigma_{\lambda'}$ in
$\P^{\ell-1}$ are all equal.

On the other hand, the Orlik-Solomon complexes for $\A'$ and $\A$ are
related by
\[
(A(\A'),\omega_{\lambda'})\cong 
(A(\A),\omega_\lambda)\otimes_\C(\C\stackrel{0}\rightarrow\C).
\]
Then the least $p$ for which $H^p(A(\A),\omega_\lambda)\neq0$ is the
same as that for which $H^p(A(\A'),\omega_{\lambda'})\neq0$.

Second, recall that 
an arrangement $\A$ in $V$ is said to be reducible if there exist subspaces
$V_1$ and $V_2$ with $V\cong V_1\oplus V_2$ and a nontrivial partition
$P_1\sqcup P_2=[n]$ for which $f_i\in V_j^*$ if and only if $i\in P_j$.
If $\A$ is reducible, write $\A=\A_1\oplus \A_2$, where $\A_j$
is the arrangement in $V_j$ of hyperplanes indexed by $P_j$.   Otherwise,
$\A$ is said to be irreducible.
\subsection{Complexes of forms}\label{subsec:forms}
Fix a central
arrangement $\A$ of $n$ hyperplanes in $V=\C^\ell$, with defining 
polynomial $Q$.  We assume that $\A$ is essential, that is, contains a 
subarrangement of $\ell$ linearly independent hyperplanes.  Recall that $R$ 
is the coordinate ring of $V$.  
The localization $R_Q$ is the coordinate ring of the hyperplane complement 
$M$.

Let $\kk=\kk(\A)=\C[a_1,\ldots,a_n]$, where $a_1,\dots,a_n$ will be 
interpreted as 
weights on the hyperplanes, and let $S=\kk\otimes R$. 
For each $p$ and $k$-algebra $T$, let $\Omega^p_{T/k}$ 
be the $T$-module of $k$-valued K\"ahler $p$-forms over $T$, 
so that $\Omega^p_{R/\C}$ and $\Omega^p_{S/\kk}$ are $\C$- and $\kk$-valued
polynomial $p$-forms on $V$, respectively.  For $T=R,S$, 
let $\Omega^p_{T/k}(*\A)=\Omega^p_{T_Q/k}$, the $T_Q$-module of
$k$-valued, rational $p$-forms with poles on the hyperplanes $\A$.
Write $\Omega^p(*\A)=\Omega^p_{R/\C}(*\A)$ for short.
Similarly, the $T$-module $\Omega^p_{T/k}(\A)$ of 
logarithmic $p$-forms with poles along $\A$ is defined by
\begin{equation}\label{eq:logforms}
\Omega^p_{T/k}(\A) = \set{\eta \in \Omega^p_{T/k}(*\A) \colon Q\eta \in  
\Omega^p_{T/k}\ \text{and}\ 
Q \d \eta \in  \Omega^{p+1}_{T/k}},
\end{equation}
and again write $\Omega^p(\A)=\Omega^p_{R/\C}(\A)$.
In particular, $\Omega^p_{T/k}(\A)=0$ if $p<0$ or $p>\ell$.  

For any $\eta\in\Omega^k(\A)$,
by definition, $Q \eta \in \Omega^k_{R/\C}$.  If $\eta$ is homogeneous, 
we say its total degree is $m$ and write $\tdeg(\eta)=m$ if
\[ 
Q\eta=\sum f_I \d x_I \ \text{and}\ m = k+\deg f_I - \deg Q = k +
\deg f_I - n. 
\] 
Let $\Omega^\hdot(\A)_m = \set{\eta \in \Omega^\hdot(\A) \mid
\tdeg(\eta)=m}$.  

For a $\Z$-graded module $N$ and integer $r$,
define the shift $N(r)$ by $N(r)_q=N_{r+q}$, for
$q\in \Z$.  Then $R(n-\ell)\cong\Omega^\ell(\A)$ 
via the map $1\mapsto Q^{-1}\d x_1\wedge\cdots\wedge \d x_\ell$.  

We recall that
$\Omega^1(\A)$ is the $R$-dual of $\derA$: see \cite[4.75]{ot}.  Moreover,
$\A$ is free if and only if $\Omega^1(\A)$ is a free $R$-module.  The
logarithmic forms themselves are self-dual: 
\begin{lemma} For each $p$, $0\le p \le \ell$, we have
$\Hom_R(\Omega^p(\A),R)\cong
\Omega^{\ell-p}(\A)(\ell-n)$.
\end{lemma}
\begin{proof}
Exterior multiplication gives a map $\Omega^p(\A)\otimes_R\Omega^{\ell-p}(\A)\to
R(n-\ell)$ from \cite[4.79]{ot}.  By comparing with the regular forms, it
is straightforward to check this is a nondegenerate pairing.
\end{proof}

The following turns out to be an interesting weakening of freeness.
\begin{definition}\label{def:tame}
Say that an arrangement $\A$ in $V$ is tame if the projective dimension
of each module of logarithmic forms is bounded by cohomological degree:
that is, 
$
{\rm pd}_R \Omega^p(\A)\leq p
$
for all $p$ with $0\leq p\leq \ell$.
\end{definition}

We will make use of several choices of differential on the graded vector 
spaces $\Omega^\hdot(*\A)$ and $\Omega^\hdot(\A)$.  First,
the exterior derivative $\d\colon \Omega^p(*\A)\to\Omega^{p+1}(*\A)$
restricts to the logarithmic forms $\Omega^\hdot(\A)$, making both 
$(\Omega^\hdot(*\A),\d)$ and $(\Omega^\hdot(\A),\d)$ ($\C$-)cochain complexes.
Also, for $(T,k)=(R,\C)$ or $(S,\kk)$, for any $\omega\in\Omega^1_{T/k}(\A)$, 
we shall denote
by $(\Omega^\hdot_{T/k}(*\A),\omega)$ and $(\Omega^\hdot_{T/k}(\A),\omega)$ 
the cochain
complexes of $T_Q$- and $T$-modules, respectively, obtained by
using (left)-multiplication by $\omega$ as a differential.
Last, for $t\in\C$, let $\nabla_t=\d+t\omega$, and $\nabla=\nabla_1$.  
As long as $\d\omega=0$, this gives a third choice of differential.

Observe that the log complex decomposes into complexes of
finite dimensional vector spaces 
\[ 
(\Omega^\hdot(\A),\omega_\lambda) =
\bigoplus_{m\in\Z} (\Omega^\hdot(\A)_m,\omega_\lambda), \ \text{resp.,}\
(\Omega^\hdot(\A),\nabla) = \bigoplus_{m\in\Z} (\Omega^\hdot(\A)_m,\nabla). 
\]

\subsection{Localizations}\label{ss:local}
If $\p$ is a prime ideal of $R$, following \cite[4.6]{ot}, let 
$X(\p)$ denote the subspace in $L(\A)$ of least dimension containing
$V(\p)$.  Then $\Omega^p(\A)_\p=\Omega^p(\A_X)_\p$ where
$X=X(\p)$: in particular, the localization $\Omega^1(\A)_\p$ is a 
free $R_\p$-module if and only if $\A_X$ is a free arrangement.

Recall a central arrangement
$\A$ is said to be {\em locally free} if $\A_X$ is free for all
$X\neq\set{0}$: see \cite{MuSc01}.  
In this case, $\Omega^1(\A)_\p$ is free for all
prime ideals not equal to the homogeneous maximal ideal $R_+$.  Since
all rank $2$ arrangements are free, the locus on which $\Omega^1(\A)_\p$
is not a free module has codimension at least $3$.

\subsection{The meromorphic ideal}
Recall that our goal is to understand
the solutions to the $\ell$ equations given by $\omega_\lambda=0$
as $\lambda\in\C^n$ varies, where the $1$-form
$\omega_\lambda$ is defined in \eqref{eq:omega}.  It is natural, then,
to consider the ``universal'' $1$-form.

\begin{definition}\label{def:Imero}
Let
\begin{equation}\label{eq:omega_a1}
\omega_\b{a}=\sum_{i=1}^n a_i \frac{\d f_i}{f_i} \in \Omega^1_{S/\kk}(\A),
\end{equation}
and let $\Imer$ be the ideal of $S_Q$ defined by the $\ell$
equations $\omega_\b{a}=0$.  We will call $\Imer$ the {\em meromorphic
ideal} of critical sets for $\A$.  
\end{definition}

In coordinates, 
if the hyperplanes of $\A$ are defined by equations
$f_i=\sum_{j=1}^\ell c_{ij}x_j$ for $1\leq i\leq n$,
then
\begin{equation}\label{def:omega_a}
\omega_\b{a}=\sum_{i,j} \frac{a_ic_{ij}}{f_i}\d x_j,
\end{equation}
and the meromorphic ideal $\Imer$ is generated by the elements
$\set{d_j\colon 1\leq j\leq \ell}$, where $d_j=\sum_{i}a_ic_{ij}/f_i$.
Thus, $\Imer$ is the image of the 
duality pairing $\ip{\der_{\kk}(S),\omega_{\b a}}$ in $S_Q$.

For $\omega_\lambda\in A^1(\A)\cong\C^n$, 
the degree-$1$ part of the Orlik-Solomon algebra, let
\begin{equation}\label{eq:Sigma}
\Sigma_\lambda=V(\omega_\lambda)\subseteq M
\end{equation}
denote the critical set of the master function $\Phi_\lambda$.  Further
let
\begin{equation}\label{eq:defSigma}
\Sigma=\Sigma(\A)=\set{(x,\omega)\in M\times A^1\colon\omega_{\b a}(x)=0},
\end{equation}
and note that $\Sigma\cong V(\Imer)$.  Denote by $\pi^*_1$, 
$\pi^*_2$ the two projections
\begin{equation}\label{eq:defpi}
\xymatrix{
V & V\times\C^n\ar[l]_{\pi^*_1}\ar[r]^{\pi^*_2} & \C^n
}
\end{equation}
induced by the inclusions of coordinate rings
$$
\xymatrix{
R\ar[r]^{\pi_1} & S & \kk\ar[l]_{\pi_2}.
}
$$

\begin{proposition}[Proposition 4.1, \cite{ot95}]\label{prop:univ}
If $\A$ is an arrangement of rank $\ell$, then 
$\Sigma$ is a codimension-$\ell$ complex manifold embedded in
$V\times\C^n$.
\end{proposition}
More precisely, one has the following.  (See \cite[Theorem~4]{HKS05} for
a 
related result.)
\begin{proposition}\label{prop:univ2}
The restriction of the projection $\pi^*_1\colon \Sigma\to M$
gives $\Sigma$ the structure of a trivial vector bundle over 
$M$ of rank $n-\ell$.
\end{proposition}
\begin{proof}
Let $W=\set{\lambda\in\C^n\colon \sum_{i=1}^n \lambda_if_i=0}$,
a codimension-$\ell$ subspace.  Now define a map
\begin{equation}\label{eq:section}
s\colon M\times W\to M\times\C^n
\end{equation}
by setting $\pi^*_1\circ 
s(x,\lambda)=x$ and $\pi^*_2\circ s(x,\lambda)=\sum_{i=1}^n\lambda_if_i(x)e_i$
for $1\leq i\leq n$, where $e_i$ denotes the $i$th coordinate vector in
$\C^n$.  Since $\sum_{i=1}^n \lambda_i\d f_i=0$ for
$\lambda\in W$, it follows from
\eqref{eq:omega} that the image of $s$ actually lies in $\Sigma$.
Since $f_i(x)\neq0$ for $x\in M$, the map $s$ is invertible:
\begin{equation}\label{eq:trivbundle}
\Sigma\cong M\times W.
\end{equation}
\end{proof}
So for each $x\in M$, the fibre ${\pi_1^*}^{-1}(x)$ is a $n-\ell$-dimensional
vector space $W$ of weights $\lambda$ for which $x\in \Sigma_\lambda$.
The fibres of the other projection, $\pi^*_2\colon 
\Sigma\to A^1$, 
are the critical sets: $\Sigma_\lambda={\pi_2^*}^{-1}(\lambda)$
for each $\lambda\in A^1$.  We can also see the limit behaviour of 
critical sets near the origin in $V$.  Let $\overline{\Sigma}$ denote
the closure of $\Sigma$ in $V\times \C^n$.
\begin{proposition}\label{prop:closure1}
If $\A$ is an irreducible arrangement, then
$$
\overline{\Sigma}\cap\big({\pi_1^*}^{-1}(0)\big)=\Big\{(0,\lambda)\in
V\times\C^n\colon
\sum_{i=1}^n \lambda_i=0\Big\}.
$$
\end{proposition}
\begin{proof}
The second coordinate of the map $s$ from \eqref{eq:section} 
lies in the hyperplane $H:=\set{\lambda\in\C^n\colon \sum_{i=1}^n\lambda_i=0}
={\rm span}(e_i-e_j\colon 1\leq i,j\leq n)$,
so the projection of $\overline{\Sigma}$ onto $\C^n$ also lies in $H$.

To show equality, let $J_s$ be the Jacobian of $s$, and use calculus
to check that the limit of the image of $J_s$ at $x=0$
contains a set of vectors which span $H$.  Reordering the
hyperplanes if necessary, suppose that
$\set{f_1,\ldots,f_{r+1}}$ form a circuit in $\A$.  By definition,
any $r$ of the set are linearly independent, and there exist 
nonzero scalars $c_1,\ldots,c_{r+1}$, for which
$\sum_{i=1}^{r+1} c_if_i=0$.  Regarding $f_{r+1}$ as a function
of $\set{f_i\colon 1\leq i\leq r}$, we have 
\begin{equation}\label{eq:pfclosure}
\frac{\partial f_{r+1}}{\partial f_i}=-c_i/c_{r+1},
\end{equation}
for each $1\leq i\leq r$.
Let $\lambda=\sum_{i=1}^{r+1}c_ie_i$: by construction, $\lambda\in W$.
Now evaluate $J_s$
at $(x,\lambda)$.  Consider partial derivatives of the
$j$th coordinate of $\pi_2^{*}\circ s$, for $1\leq j\leq r+1$:
$$
\frac{\partial}{\partial{f_i}}\lambda_j f_j(x)=
\begin{cases}
c_i&\text{if $j=i$;}\\
c_{r+1}(-c_i/c_{r+1})&\text{if $j=r+1$;}\\
0&\text{otherwise.}
\end{cases}
$$
Since the coefficients $c_i$ are nonzero, this implies $(0,e_i-e_{r+1})$
is in the limit of the image of $J_s$
for each $1\leq i\leq r$.  By linearity, $(0,e_i-e_j)\in\overline{\Sigma}$
whenever the hyperplanes indexed by $i$ and $j$ are contained in a 
common circuit.  Since $\A$ is irreducible, its underlying matroid is
connected, so any two hyperplanes are 
contained in a common circuit.  It follows that
the closure of $\Sigma$ over $x=0$ equals $H$.
\end{proof}

\subsection{The logarithmic ideal}
The critical variety $\Sigma$ becomes more tractible when 
it is extended to the affine space $V$.  We indicate two natural
ways to do this which turn out to coincide.

As in \cite{ot95}, we may apply the logarithmic derivations $\derA$
to obtain critical equations in the polynomial ring $S$.  Let
$\Ilog=\Ilog(\A)=(\ip{\der_\kk(\A),\omega_\b{a}})$ be the image of the duality
pairing.
It follows from \eqref{eq:defDer} that $\Ilog$ is actually an ideal in
the
polynomial ring $S$, rather than just the localization $S_Q$.  We
will call $\Ilog(\A)$ the {\em logarithmic ideal} of critical sets for $\A$.

If the arrangement $\A$ is free, one can write generators of $\Ilog$
explicitly
as follows.  First, $\derA$ is a free $R$-module with some homogeneous
basis $\set{D_1,\ldots,D_\ell}$.  Then 
\begin{equation}\label{eq:Deqs}
D_i=\sum_{j=1}^{\ell}g_{ij}\partial/\partial x_j
\end{equation}
for some polynomials $\set{g_{ij}}$.  Let $m_i$ denote the (total)
degree
of $D_i$, for each $i$, ordering $D_1,\dots,D_\ell$ so that $m_1\leq\cdots\leq 
m_\ell$.  We may assume $D_1$ is the Euler derivation, and $m_1=0$.
The numbers $\set{m_i}$ are classically called the exponents of $\A$.
\begin{proposition}
If $\A$ is a free arrangement, then 
the ideal $\Ilog$ has homogeneous generators in the exponents of $\A$.
\end{proposition}
\begin{proof}
Apply the derivations \eqref{eq:Deqs} to $\omega_{\b a}$, introduced in
\eqref{def:omega_a}.  Explicitly,
\begin{equation}\label{eq:defI3}
\Ilog=\Big(\sum_{j=1}^\ell g_{ij}d_j:1\leq
i\leq\ell\Big).
\end{equation} 
Since each $d_j$ is a rational function with simple poles
and each $g_{ij}$ has degree
$m_i+1$, the polynomial $\sum_{j=1}^\ell g_{ij}d_j$ is homogeneous
of degree $m_i$.
\end{proof}
If $\A$ is not free, only the generators of $\derA$ in minimal degree
are easily understood.  In particular, if $\A$ is irreducible, 
the Euler derivation generates $\derA_0$, which gives the following.
\begin{proposition}\label{prop:Ideg0}
If $\A$ is an irreducible arrangement, then the degree $0$ part of
$I$ is generated by $\sum_{i=1}^n a_i$.
\end{proposition}

\begin{theorem}\label{thm:closure}
For any arrangement $\A$, 
$V(I(\A))$ is the closure of $\Sigma(\A)$ in $V\times\C^n$.  
\end{theorem}
Accordingly, we will write $\OSigma=V(I)$.  We 
defer the proof to \S\ref{ss:pfclosure}.
\begin{corollary}\label{cor:closure}
For any arrangement $\A$ of rank $\ell$, 
the variety $\OSigma$ is irreducible of codimension $\ell$.
\end{corollary}
\begin{proof}
By Theorem~\ref{thm:closure}, the vanishing ideal of $\OSigma$ is
$\rad(I)$.  This is the contraction of $I_Q=\Imer$, and $\Imer$ is 
prime by Proposition~\ref{prop:univ}.  It follows that the radical of $I$
is also prime.
\end{proof}
In general, the ideal $I$ need not be radical (Example~\ref{ex:ER}).
However, we will see that if $\A$ is tame, then $I$ is actually prime
(Corollary~\ref{cor:Iprime}.)

\subsection{A naive ideal}
For purposes of comparison, let $\IQ=(Qd_j\colon 1\leq j\leq \ell)$, 
the ideal of $S$ obtained by clearing denominators in 
Definition~\ref{def:Imero}.  From a geometric point of view, this ideal
should be replaced by an ideal quotient by $Q$.
It turns out that doing so recovers the logarithmic ideal.  We note
that a closely related result appears in the algorithm of \cite{HKS05}:
in that setting, the weights $a_i$ are specialized to natural numbers,
while the polynomial $Q$ is generalized to an arbitrary homogeneous ideal.

\begin{proposition}\label{prop:sat}
For any arrangement $\A$, we have
$(\IQ:Q)=\Ilog$. 
\end{proposition}
\begin{proof} 
By definition, $(I':Q)=\set{s\in S\colon sQ\in I'}$.
To show
$\Ilog\subseteq (\IQ:Q)$, write any $\theta\in\der_{\C}(\A)$ as
$\theta=\sum_{j=1}^\ell r_j\partial/\partial x_j$ for some coefficients
$r_j\in S$. Then 
\[
Q\ip{\theta,\omega_\b{a}}=Q\sum_{j=1}^\ell r_jd_j=\sum_{j=1}^\ell
r_j(Qd_j)\in \IQ,
\]
so $\ip{\theta,\omega_\b{a}}\in (\IQ:Q)$.
To show the other inclusion, suppose $f\in (\IQ:Q)$.  We may write
$$
fQ=\sum_{j=1}^\ell r_j Qd_j
$$
for some polynomials $r_j\in S$; that is, $f=\sum_j r_j d_j\in S$.

Form the derivation $\theta=\sum_{j=1}^\ell r_j\frac{\partial}{\partial
x_j}$. Since $\ip{\theta,\omega_\b{a}}=f$, to show $f\in \Ilog$ it is
enough
to show $\theta\in\derA\otimes S$. In turn, since $\ip{\theta,\d
f_i}=\theta(f_i)$, we need to prove $\ip{\theta,\d f_i}\in(f_i)$ for
each $i$, $1\leq i\leq n$ (by \cite[Prop.~4.8]{ot}).

For this, use \eqref{eq:omega_a1} to write
\[
fQ= \ip{\theta,\omega_{\b a}}Q
=\sum_{i}\ip{\theta,\d f_i}a_i Q/f_i.
\]
Since $fQ\in(Q)$, the image of $fQ$ under the map $S\rightarrow
S/(f_1)\times\cdots\times S/(f_n)$ is zero.  Since $Q/f_i$ is
divisible by all $f_j$, $j\neq i$, it follows the image of
$\ip{\theta,\d f_i}a_i Q/f_i$ is also zero for each $i$. Since $a_i
Q/f_i\neq0$ in $S/(f_i)$, and $(f_i)$ is a prime ideal, $\ip{\theta,\d
f_i}=0$ in $S/(f_i)$, and $f\in I$ as claimed.
\end{proof} 
\subsection{Complete intersections}
It follows from Proposition~\ref{prop:univ} together with
Corollary~\ref{cor:closure} that the codimension
of $I$ and $\Imer$ both equal $\ell$.  Since $S$ and therefore $S_Q$
are
Cohen-Macaulay, the depth of $I$ and $\Imer$ are also both $\ell$, see
\cite[Theorem~18.7]{eisenbook}.  
Since $\Imer$ is generated by $\ell$ elements of $S_Q$, we obtain the following.
\begin{lemma}\label{lem:ci}
The ideal $\Imer$ is a complete intersection. 
\end{lemma}
The logarithmic critical set ideal behaves more subtly.
\begin{theorem}\label{th:ci} 
The ideal $I$ is a complete intersection if and only if $\A$ is free.
\end{theorem} 
\begin{proof} 
If $\A$ is free, then $I$ is generated by $\ell$ elements,
\eqref{eq:defI3}.  Since $I$ has codimension $\ell$, by 
Corollary~\ref{cor:closure}, $I$ is a complete intersection.  On the other hand, 
suppose that $I$ is a complete intersection.  Then it has
some $\ell$ homogeneous generators $f_1,f_2,\ldots,f_\ell$.  For each $i$,
$1\leq i\leq \ell$, let $\theta_i\in\der_C(\A)$ be a derivation for which 
$\theta_i(\omega_{\b a})=f_i$.  By definition of $I$, the module
$\der_\kk(\A)$ is generated by $\theta_1,\ldots,\theta_\ell$.

Since the number of generators is equal to its rank and $S$ is a domain,
$\der_\kk(\A)$ is a free $S$-module.  It follows $\der_\C(\A)$ is a
flat $R$-module.  Since $\der_\C(\A)$ is a finitely-generated graded
module, it is free.
\end{proof}
\begin{example}
The arrangement $X_3$, defined by $Q=xyz(x+y)(x+z)(y+z)$, is not
free.  The ideal $I$ is minimally generated by $4$ generators, not $3$.
\end{example}
\begin{example}[Pencils]\label{ex:pencil}
An arrangement $\A$ of $n$ lines in $\C^2$ is free 
(\cite[Example 4.20]{ot}).  $\der_\C(\A)$ has a basis $\set{D_1,D_2}$ where
$D_2=Q/f_1(\frac{\partial f_1}{\partial x_1}\frac{\partial}{\partial x_2} -
\frac{\partial f_1}{\partial x_2}\frac{\partial}{\partial x_1})$, 
and $D_1$ is the Euler derivation.
Then $I$ is a complete intersection generated in degrees $0,n-2$:
$$
I=\left(\sum_{H\in\A}a_H, \sum_{i=2}^n a_i\frac{Q}{f_1f_i}\left(
\frac{\partial f_1}{\partial x_1}\frac{\partial f_i}{\partial x_2} -
\frac{\partial f_1}{\partial x_2}\frac{\partial f_i}{\partial x_1}\right)
\right).
$$
\end{example}

\section{Koszul complexes}\label{sec:koszul}
In this section, we indicate how to determine the codimension of a 
critical set using the complexes of forms from \S\ref{subsec:forms}.
The key idea is that $(\Omega^\hdot_{S/C}(*\A),\omega_{\b a})$ 
is the Koszul complex for the defining ideal of $\Sigma$.  In the
tame case, $(\Omega^\hdot_{S/C}(\A),\omega_{\b a})$ is a resolution,
not necessarily free, of the defining ideal of $\OSigma$.

\subsection{Meromorphic forms}
Recall from \S\ref{subsec:forms} that
$\Omega^\hdot(*\A)$ is the space of meromorphic forms with poles on the
hyperplanes.
We may regard this as the exterior algebra on $R_Q$.
Then $(\Omega^\hdot_{S/\kk}(*\A),\omega_\b{a})$ is the Koszul
complex of the generators $\set{d_j}$ of $\Imer$, by definition of
$\omega_{\b a}$.  Since the depth of $\Imer$ is equal to $\ell$, we obtain
the following.
\begin{proposition}\label{prop:mer}
If $\A$ is a central, essential arrangement,
\begin{equation}\label{eq:res}
0\rightarrow\Omega^0_{S/\kk}(*\A)
\stackrel{\omega_\b{a}}{\rightarrow}\Omega^1_{S/\kk}(*\A)
\stackrel{\omega_\b{a}}{\rightarrow}\cdots
\stackrel{\omega_\b{a}}{\rightarrow} \Omega^\ell_{S/\kk}(*\A)\rightarrow
S_Q/\Imer\rightarrow0 
\end{equation} 
is an exact complex of $S_Q$-modules.
\end{proposition}

\begin{definition}\label{def:Rlambda}
For $\lambda\in\C^n$, let
$R_\lambda=R$, regarded as an $S$-module via the homomorphism that maps
$a_i$ to $\lambda_i$, for each $i$, $1\leq i\leq n$.
\end{definition}

\begin{corollary}\label{cor:torQ} 
For any $\lambda\in\C^n$ and $0\leq p\leq \ell$, 
\begin{equation}\label{eq:torQ}
H^p(\Omega^\hdot(*\A),\omega_\lambda)\cong\Ext_{S_Q}^{p}(S_Q/
\Imer, (R_\lambda)_Q).
\end{equation}
If the critical set $\Sigma_\lambda$ is nonempty, then the codimension of 
$\Sigma_\la$ 
is the smallest $p$ for which $H^p(\Omega^\hdot(*\A),\omega_\lambda)\neq0$. 
\end{corollary}
\begin{proof} 
By Proposition~\ref{prop:mer}, the complex \eqref{eq:res} is a 
free resolution of $S_Q/\Imer$.  Now applying $\Hom_{S_Q}(-,(R_\lambda)_Q)$
to the complex \eqref{eq:res} gives $(\Omega^\hdot(*\A),\omega_\lambda)$,
since Koszul complexes are self-dual.  Taking cohomology, we obtain
\eqref{eq:torQ}.

Now $\Imer\otimes (R_\lambda)_Q$ is the vanishing ideal of $\Sigma_\la$,
and $(\Omega^\hdot(*\A),\omega_\lambda)$ its Koszul complex.  
Since $R_Q$ is Cohen-Macaulay, the codimension of the ideal 
is equal to its depth, which is the least 
$p$ for the cohomology of the Koszul complex \eqref{eq:torQ} is nonzero.
(In particular, if the critical set is 
empty, then $H^p(\Omega^\hdot(*\A), \omega_\lambda)=0$ for all $p$.) 
\end{proof}

\subsection{Logarithmic forms} 
Analogous statements hold for the complex $\Omega^\hdot(\A)$, provided that 
the arrangement $\A$ is tame.  The advantage is that, since $\A$ is a
central arrangement, $\Omega^\hdot(\A)$ is a graded $R$-module.  

Since localization is exact and $\Omega^\hdot(\A)_Q\cong \Omega^\hdot(*\A)$,
Corollary~\ref{cor:torQ} yields the following.
\begin{proposition}
For any $\lambda\in\C^n$, if $\Sigma_\lambda$ is nonempty, then the codimension 
of $\Sigma_\la$ 
is the least $p$ for which 
$$
H^p(\Omega^\hdot(\A),\omega_\lambda)_Q\neq0.
$$
\end{proposition}

For tame arrangements, it is possible to make a more precise analysis.
The proof of the following is deferred to the next section.
\begin{theorem}\label{th:two} 
If $\A$
is free, then the ``universal'' log-complex is a free resolution of
$(S/I)(n-\ell)$ as a graded $S$-module.  More generally,
for any tame arrangement $\A$, the complex
\begin{equation}\label{eq:logcplx}
0\rightarrow\Omega^0_{S/\kk}(\A)
\stackrel{\omega_\b{a}}{\rightarrow}\Omega^1_{S/\kk}(\A)
\stackrel{\omega_\b{a}}{\rightarrow}\cdots
\stackrel{\omega_\b{a}}{\rightarrow} \Omega^\ell_{S/\kk}(\A)\rightarrow
(S/I)(n-\ell)\rightarrow0 
\end{equation}
is exact. 
\end{theorem} 

By analogy with Corollary~\ref{cor:torQ}, we have the following.
\begin{corollary}
Suppose $\A$ is tame.  
For any $\lambda\in\C^n$ and $0\leq p\leq \ell$, then
\begin{equation}\label{eq:tor}
H^p(\Omega^\hdot(\A),\omega_\lambda)\cong\tor^S_{\ell-p}(S/I,R_\lambda)(n-\ell),
\end{equation}
where $R_\lambda$ is the specialization from
Definition~\ref{def:Rlambda}.
\end{corollary}
\begin{proof}
If $\A$ is free, the statement follows directly.  More generally,
since the complex \eqref{eq:logcplx} is exact, the first hyper-Tor
spectral sequence degenerates:
\begin{equation}\label{eq:hypertor}
\tor^S_\hdot(S/I,M)(n-\ell)\cong 
\mathbf{Tor}^S_\hdot(\Omega^{\ell-\hdot}_{S/\kk}(\A),M),
\end{equation}
for any $S$-module $M$.  Since, we also have
$\tor^S_q(\Omega^p_{S/\kk}(\A),R_\lambda)=
\tor^R_q(\Omega^p_{R/\C}(\A),R)=0$ for $q>0$, by flat base change,
the second hyper-Tor spectral sequence also degenerates, giving the
isomorphism claimed.
\end{proof}

Recall that a complete intersection is an example of a Cohen-Macaulay
ring, for which we refer to \cite{eisenbook}.  Theorem~\ref{th:ci} can
be extended as follows.  
\begin{theorem}\label{thm:CM}
If $\A$ is a tame arrangement, then the affine coordinate ring $S/I$ 
of $\OSigma$ is Cohen-Macaulay.
\end{theorem}
Note that the coordinate ring $S/I$ is not Cohen-Macaulay for all arrangements 
(Example~\ref{ex:ER}).
\begin{proof}
Since $\OSigma$ has codimension $\ell$ (Corollary~\ref{cor:closure}),
the depth of $I$ is $\ell$.  It follows that ${\rm pd}_S(S/I)\geq\ell$.
The ring $S/I$ is Cohen-Macaulay if and only if this is an equality.
Using the isomorphism \eqref{eq:hypertor} for $M=\C$, we have a 
spectral sequence
$$
E^1_{pq}=\tor_q^S(\Omega^{\ell-p}_{S/\kk}(\A),\C)\Rightarrow
\tor^S_{p+q}(S/I,\C)(n-\ell).
$$
The tame hypothesis is equivalent to having $E^1_{pq}=0$ for $p+q>\ell$.  
It follows that
${\rm pd}_S(S/I)\leq \ell$, as required.
\end{proof}
This allows a sharpening of Theorem~\ref{thm:closure}.
\begin{corollary}\label{cor:Iprime}
If $\A$ is a tame arrangement, then $I$ is the vanishing ideal of $\OSigma$.
In particular, $I$ is prime.
\end{corollary}
\begin{proof}
By Theorem~\ref{thm:closure}, the vanishing ideal of $\OSigma$ is 
$\rad(I)$.  This is prime, by Corollary~\ref{cor:closure}.
Suppose $\A$ is tame.  By Theorem~\ref{thm:CM}, the ideal $I$ has no
embedded primes, so $I$ is primary.  Since $Q\not\in\rad(I)$, it follows
$(I:Q)=I$.  Then $I$ is the contraction of $I_Q=\Imer$ under the inclusion
$S\hookrightarrow S_Q$, so $I=\rad(I)$.
\end{proof}

Now suppose $\Phi_\lambda$ is a master function with
$\omega_\lambda=\d\log\Phi_\lambda$.  Then
$S/I\otimes_S R_\lambda = R/I_\lambda$, where $I_\lambda$ is
the ideal generated by $\ip{\derA,\omega_\lambda}$.  Let
\begin{equation}\label{eq:OSigma}
\OSigma_\lambda = V(I_\lambda) \subseteq\C^\ell.
\end{equation}
Clearly $\OSigma_\lambda\cap M=\Sigma_\lambda$.  However,
it is not the case in general that $\OSigma_\lambda$ is the
closure of $\Sigma_\lambda$ (see Example~\ref{ex:monomial deletions}.)
The next result prepares for our main Theorem~\ref{thm:res2crit}.
\begin{proposition}\label{prop:tamecodim}
If $\A$ is a tame arrangement, then
the codimension of $\OSigma_\lambda$ 
is the smallest $p$ for which $H^p(\Omega^\hdot(\A),\omega_\lambda)\neq0$,
provided that either $\A$ is free, $\A$ has rank $3$, or $p\leq2$.
\end{proposition}
\begin{proof}
The codimension of $\OSigma_\lambda$ is equal to the depth of $I_\lambda$,
or equivalently the depth of $I$ on $R_\lambda$, 
which is the least $p$ for which 
$\Ext_S^p(S/I,R_\lambda)\neq0$.  By the same argument as in
\eqref{eq:hypertor}, using Theorem~\ref{th:two} with
hyper-Ext gives a spectral sequence
$$
E_1^{pq}=\Ext^q_S(\Omega_{S/\kk}^{\ell-p}(\A),R_\lambda)
\Rightarrow \Ext_S^{p+q}(S/I,R_\lambda),
$$
suppressing the degree shift.
Then $E^{pq}_1\cong\Ext_R^q(\Omega_{R/\C}^{\ell-p}(\A),R)$, and in particular
$E^{p0}_1\cong \Omega_{R/\C}^p(\A)$ by self-duality.  Consequently,
$E^{p0}_2\cong H^p(\Omega^\hdot(\A),\omega_\lambda)$.  If $\A$ is free,
$E^{pq}_1=0$ for $q>0$, and the conclusion follows.

In general, since $\Omega^\ell$
is a free module, $E^{0q}_1=0$ for $q>0$, which means
$E_2^{p0}=E_\infty^{p0}$ for $p\leq 2$.  That is,
\begin{equation}\label{eq:extiso}
\Ext^p_S(S/I,R_\lambda)\cong H^p(\Omega^\hdot(\A),\omega_\lambda)
\end{equation}
for $0\leq p\leq 2$.  The claims for rank-$3$ arrangements
and codimensions $p\leq2$ follow.
\end{proof}

\subsection{Proof of Theorem~\ref{th:two}}\label{ss:pftwo}
The purpose of this section is to show that the complex \eqref{eq:logcplx}
is exact if the arrangement $\A$ is tame.  We begin with a reduction
to irreducible arrangements.  Suppose $\A=\A_1\oplus \A_2$, and let
$S_j$ and $\kk_j$ be corresponding coordinate rings for $j=1,2$, and
let $\omega_{{\b a}_j}\in\Omega^1_{S_j/\kk_j}(\A_j)$ as defined in
\eqref{eq:omega_a1}.  The following lemma is routine and we omit the
proof: a similar result appears as \cite[Proposition 4.14]{ot}.
\begin{lemma}\label{lem:decomp}
If $\A=\A_1\oplus \A_2$, the decomposition induces an isomorphism of
cochain complexes
$$
(\Omega^\hdot_{S/\kk}(\A),\omega_{\b a})\cong
(\Omega^\hdot_{S_1/\kk_1}(\A_1),\omega_{{\b a}_1})
\otimes_\C
(\Omega^\hdot_{S_2/\kk_2}(\A_2),\omega_{{\b a}_2}).
$$
Moreover, $\Sigma(\A)\cong\Sigma(\A_1)\times\Sigma(\A_2)$ and
$I(\A)=S_2I(\A_1)+S_1I(\A_2)$.
\end{lemma}  
 
We now argue induction on the number of hyperplanes.  Clearly if $\abs{\A}=1$,
the arrangement is free, and \eqref{eq:logcplx} is exact.  We may 
assume $\A$ is irreducible: if not, by Lemma~\ref{lem:decomp},
the complex \eqref{eq:logcplx} decomposes as a tensor product of
complexes for arrangements with strictly fewer hyperplanes.  By induction
and the K\"unneth formula, then, \eqref{eq:logcplx} is exact.

Since an $S$-module $N$ is zero if and only if $N_\m=0$ for all maximal ideals
$\m$, we wish to show that
the localization of \eqref{eq:logcplx} is exact for every $\m$.  It is 
enough to consider
just those ideals $\p=S\m$, for maximal ideals $\m$ of $R$.

Let $X=X(\m)$, in the notation of \S\ref{ss:local}.  
First, consider the case where $X\neq0$.  By 
assumption, $\A$ is essential, so some hyperplane does not contain $X$.
Without loss of generality, assume $X\not\subseteq\ker f_n$, and let $\A'$ 
denote the arrangement obtained from $\A$ by deleting the last hyperplane.

Since $\A$ is irreducible, $\A'$ is an essential arrangement in $V$.
Let $\kk'=\C[a_i\colon 1\leq i\leq n-1]$ and $S'=\kk'\otimes_\C R$.
Similarly, let $\omega_{{\b a}'}=\sum_{i=1}^{n-1}a_i\d f_i/f_i$.
Consider, for $0\leq p\leq\ell$, the inclusion
$$
i\colon \Omega^p_{S'/\kk'}(\A')\otimes_\C\C[a_n]
\hookrightarrow\Omega^p_{S/\kk}(\A).
$$
Since $f_n$ is a unit in $S_\p$, the map $i$ localizes to an isomorphism
of $S_\p$-modules.  Since $\A$ is irreducible, we may write
$
f_n=\sum_{i=1}^{n-1}c_i f_i
$
for some scalars $c_i\in\C$.

Then we have an isomorphism of cochain complexes
\begin{eqnarray}\nonumber
(\Omega^\hdot_{S/\kk}(\A),\omega_{\b a})_\p &\cong &
(\Omega^\hdot_{S'/\kk'}(\A')\otimes_\C \C[a_n],\omega_{{\b a}'}+
a_n\frac{\d f_n}{f_n})_\p\\
&\cong &
(\Omega^\hdot_{S'/\kk'}(\A')\otimes_\C \C[a_n],\eta)_\p,\label{eq:twisted}
\end{eqnarray}
where the differential is given by multiplication by 
$$
\eta=\sum_{i=1}^n (a_i+\frac{c_i f_i a_n}{f_n})\frac{\d f_i}{f_i}.
$$
Define a homomorphism $\phi\colon S_{f_n}\to S_{f_n}$ 
by setting
$\phi(x_i)=x_i$ for $1\leq i\leq \ell$ and $\phi(a_i)=a_i-c_if_ia_n/f_n$
for $1\leq i\leq n-1$, and $\phi(a_n)=a_n$.  Note $\phi$ is an isomorphism,
so it localizes to an isomorphism of local rings $S_\p\to S_{\p'}$, where
$\p'=(\phi^{-1})^*(\p)$.

By construction,
$\phi$ induces an isomorphism on forms with $\phi(\eta)=\omega_{{\b a}'}$,
taking \eqref{eq:twisted} to the cochain complex
$$
(\Omega^\hdot_{S'/\kk'}(\A')\otimes_\C \C[a_n],\omega_{{\b a}'})_{\p'}.
$$
The cohomology of this complex is concentrated in top degree, by induction,
so the complex \eqref{eq:logcplx} is exact at $\m$: in fact, we have
\begin{equation}\label{eq:localiso}
\big(S/\Ilog(\A)\big)_\p\cong\big(S/S\Ilog(\A'))_{\p'}.
\end{equation}

It remains to consider the case where $X=0$, i.e., $\m=R_+$.  
From the previous argument, \eqref{eq:logcplx} is
exact at all other maximal primes, so some power
of $R_+$ annihilates $H^q:=H^q(\Omega^\hdot_{S/\kk}(\A),\omega_{\b a})$, 
for $0\leq q<\ell$.  The following lemma shows that power must be zero,
completing the proof of exactness.
\begin{lemma}
Suppose $\A$ is a tame arrangement. Let $q$ be the least integer
for which $H^q\neq0$.  If $q<\ell$, then $H^q$ is $R_+$-saturated.
\end{lemma}
\begin{proof}
Equivalently, we wish to show that the local cohomology group
$H^0_{R_+}(H^q)=0$.  For this, consider the two hypercohomology 
spectral sequences of local cohomology, writing $\Omega^\hdot$ in 
place of $\Omega^\hdot_{S/\kk}(\A)$:
\begin{eqnarray*}
\leftidx{'}{E}{_2^{pq}}=H^p(H^q_{R_+}(\Omega^\hdot))&\Rightarrow&
{\mathbb H}^{p+q}_{R_+}(\Omega^\hdot),\quad\text{and}\\
\leftidx{''}{E}{_2^{pq}}=H^p_{R_+}(H^q(\Omega^\hdot))&\Rightarrow&
{\mathbb H}^{p+q}_{R_+}(\Omega^\hdot).
\end{eqnarray*}
The tame hypothesis implies that $H^q_{R_+}(\Omega^p)=0$ for $0\leq
q<\ell-p$.  Then, from the first spectral sequence, we obtain
${\mathbb H}^k_{R_+}(\Omega^\hdot)=0$ for $0\leq k<\ell$.

On the other hand, consider the least $q$ for which 
$H^q=H^q(\Omega^\hdot)\neq0$.  Then if $q<\ell$, 
we must have $\leftidx{''}{E}{_\infty^{0q}}=\leftidx{''}{E}{_2^{0q}}=0$.
So $H^0_{R_+}(H^q)=0$, as required.
\end{proof}
\begin{remark}
The hypothesis that $\A$ is tame was required only to show that the complex
\eqref{eq:logcplx} was exact when localized at $R_+$; other localizations
followed by induction.  Theorem~\ref{th:two} can then be extended slightly
as follows.
\end{remark}
\begin{theorem}
If $\A$ is an essential arrangement for which all proper subarrangements
$\A_X$ are tame, then the complex of coherent sheaves
on $\P^{\ell-1}\times\P^{n-1}$
$$
0\rightarrow \omsheaf^0_{S/\kk}(\A)\rightarrow \omsheaf^1_{S/\kk}(\A)
\rightarrow
\cdots\rightarrow \omsheaf^\ell_{S/\kk}(\A)\rightarrow 
{\mathcal O}_{\P\OSigma}(n-\ell)\rightarrow 0
$$
is exact.
\end{theorem}
\subsection{Proof of Theorem~\ref{thm:closure}}\label{ss:pfclosure}
The argument that the variety of the logarithmic ideal $I(\A)$ equals
the closure of $\Sigma(\A)$ is parallel to the proof of Theorem~\ref{th:two},
so we include it here to avoid unnecessary repetition.  We note, however,
that the arrangement $\A$ is not assumed to be tame here.
\begin{proof}
Again, argue by induction on $n$, the number of hyperplanes.
If $n=1$, then
$\OSigma=V(I)=\set{(0,0)}$.  If $n>1$, it suffices to 
consider irreducible arrangements, using the induction hypothesis and
Lemma~\ref{lem:decomp}.  Clearly $\Sigma\subseteq V(I)$.  If 
$(x,\lambda)\in V(I)-\Sigma$, we argue that it has a neighborhood
that intersects $\Sigma$.

First, consider the case where $x=0$.  Since $\A$ is irreducible, by
Proposition~\ref{prop:Ideg0}, $V(I)$ is given in a neighborhood of
$x=0$ by the equation $\sum_{i=1}^n a_i=0$.  Comparing with 
Proposition~\ref{prop:closure1} establishes the claim.

Otherwise, since $\A$ is assumed to be essential, we assume
again that the last hyperplane of $\A$ does not contain the point $x$.  
Let $\A'$ denote the deletion, following the notation of \S\ref{ss:pftwo}.
From \eqref{eq:localiso},
$\phi^*$ gives a homeomorphism between neighborhoods of
$(x,\lambda)\in V(I(\A))$ and $(x,\lambda')\in V(I(\A'))\times\C$,
where $\lambda'_i=\lambda_i+c_if_i(x)\lambda_n/f_n(x)$ for $1\leq i\leq n-1$
and $\lambda'_n=\lambda_n$.  By the induction hypothesis, the neigborhood
of $(x,\lambda')$ meets $\Sigma(\A')\times\C$, which means the 
neighborhood of $(x,\lambda)$ in $V(I)$ meets $\Sigma(\A)$, as required.
\end{proof}

\section{Resonant $1$-forms have high-dimensional zero loci}
\label{sec:res2crit}

The purpose of this section is to establish the following result.
\begin{theorem}\label{thm:res2crit} 
Let $\A$ be a tame
arrangement of $n$ hyperplanes in $V$.  If $\lambda \in \C^n$ is a
vector of weights for which $H^p(A(\A),\omega_\lambda) \neq 0$,
then the codimension of the critical set $\OSigma_\lambda$ 
is at most $p$, provided either $\A$ is free or $p\leq 2$.
\end{theorem}
Nonzero $\lambda$ for which $H^1(A,\omega_\lambda)\neq0$
have been studied extensively: see \cite{FY07}.  For such $\lambda$,
if $A$ is tame, then we see $\OSigma_\lambda$ is a hypersurface.

Since rank $3$ arrangements are tame~\cite{WiensYuz}, we also find:
\begin{corollary}
If $\A$ has rank $3$ and $\lambda\in\C^n$ is a collection of weights for which
$H^p(A(\A),\omega_\lambda)\neq0$, then the codimension of $\OSigma_\lambda$
is at most $p$.
\end{corollary}

To prove the theorem, we first show that 
resonance in dimension $p$ implies that the
cohomology of the log complex $\Omega^\hdot(\A)$, with differential 
$\nabla=\d+\omega_\la$, is also nontrivial in dimension $p$. 

\begin{proposition} \label{prop:mono} 
For each $t\in\C^*$, the inclusion
$(A(\A),t \omega_\lambda) \to (\Omega^\hdot(\A),\nabla_t)$ induces a
monomorphism $H^\hdot(A(\A),t \omega_\lambda) \to
H^\hdot(\Omega^\hdot(\A),\nabla_t)$. If $H^p(A(\A),\omega_\lambda) \neq 0$, 
then $H^p(\Omega^\hdot(\A),\nabla)\neq 0$. 
\end{proposition} 
\begin{proof} For
$t\neq 0$, let $\nabla_t = \d+t \omega_\lambda$.  For $t$ sufficiently
small, the inclusion $(A(\A),t \omega_\lambda) \hookrightarrow
(\Omega^\hdot(*\A),\nabla_t)$ is a quasi-isomorphism, by 
\cite[Theorem 4.6]{SV}.
Consequently, the sequence of inclusions \[ (A(\A),t \omega_\lambda)
\hookrightarrow (\Omega^\hdot(\A),\nabla_t) \hookrightarrow
(\Omega^\hdot(*\A),\nabla_t) \] implies that the map \[ H^\hdot(A(\A),t
\omega_\lambda) \to H^\hdot(\Omega^\hdot(\A),\nabla_t) \] in cohomology is a
monomorphism.  Since $H^\hdot(A(\A),t
\omega_\lambda)=H^\hdot(A(\A),\omega_\lambda)$, if
$H^p(A(\A),\omega_\lambda)\neq 0$, then $H^p(\Omega^\hdot(\A),\nabla_t) \neq
0$.  The results then follows from the upper semicontinuity with respect
to $t$ of $H^p(\Omega^\hdot(\A),\nabla_t)$. \end{proof}

In light of this result, to prove Theorem \ref{thm:res2crit}, it
suffices to show that the nonvanishing of $H^p(\Omega^\hdot(\A),\nabla)$
implies that of $H^p(\Omega^\hdot(\A),\omega_\lambda)$.  For this, we will
use a spectral sequence, following Farber \cite{Farber2001,farber}.

If $C$ is a cochain complex equipped with two differentials $d$ and
$\delta$ satisfying $d \circ \delta+\delta\circ d=0$, then for each $t
\in \C$, $(C,d+t \delta)$ is a cochain complex.  In \cite{farber},
Farber constructs a spectral sequence converging to the cohomology
$H^\hdot(C,d+t\delta)$, with $E_1$ term given by $E_1^{p,q} = H^{p+q}(C,d)$
for all $q \ge 0$, and $d_1:H^{p+q}(C,d) \to H^{p+q+1}(C,d)$ induced by
$\delta$.  For $r$ sufficiently large, the differential $d_r$ vanishes,
and $E^{p,q}_\infty \cong H^{p+q}(C,d+t \delta)$ for all but finitely
many $t\in \C$, see \cite[\S10.8]{farber}.

For each $m \in \Z$, we use this construction to analyze the cohomology
of the complex 
\begin{equation} \label{eq:om} 
(\Omega^\hdot(\A)_m,\nabla_t)=
(\Omega^\hdot(\A)_m,\d+t \omega_\lambda). 
\end{equation}

In many cases, the monomorphism of Proposition~\ref{prop:mono} is
actually an isomorphism.  In \cite{WiensYuz}, Wiens and Yuzvinsky
show that if $\A$ is a tame arrangement (Definition~\ref{def:tame}),
then $H^\hdot(\Omega^\hdot(\A),\d)\cong A(\A)$.
That is, the logarithmic forms compute the cohomology of the complement.

\begin{proposition} \label{prop:AOqi}
Suppose that  $\A$ is a tame arrangement. If $m \neq 0$, then we have
$H^\hdot(\Omega^\hdot(\A)_m,\d+t \omega_\lambda)=0$.  Furthermore, for $m=0$, the
inclusion $(A(\A),t \omega_\lambda) \hookrightarrow (\Omega^\hdot(\A)_0,\d+t
\omega_\lambda)$ induces an isomorphism in cohomology for all but
finitely many $t$. 
\end{proposition} 
\begin{proof} 
By the main theorem of \cite{WiensYuz},
$H^\hdot(\Omega^\hdot(\A),\d)=A(\A)$.
Consequently, in the Farber spectral sequence for the complex
\eqref{eq:om}, we have 
\[ 
E_1^{p,q}=H^{p+q}(\Omega^\hdot(\A)_m,\d)= \begin{cases}
0&\text{for $m\neq 0,$}\\ A^{p+q}(\A)&\text{for $m=0$.} \end{cases} 
\]
The first assertion follows immediately.  For $m=0$, since
multiplication by 
$\omega_\lambda$ induces the differential 
$d_1\colon E_1^{p,q} \to E_1^{p+1,q}$, 
the $E_2$-term of
the spectral sequence is $E_2^{p,q}=H^{p+q}(A(\A),\omega_\lambda)$.  By
Proposition \ref{prop:mono}, the vector space $E_2^{p,q}$ is a subspace
of $E_\infty^{p,q}=H^{p+q}(\Omega^\hdot(\A)_0,\nabla_t)$ for large $p$. The
result follows.  
\end{proof}

\begin{proof}[Proof of Theorem \ref{thm:res2crit}]
It suffices to show that if $H^p(A(\A),\omega_\lambda) \neq 0$  does not
vanish, then 
$H^p(\Omega^\hdot(\A),\omega_\lambda)\neq 0$ as well.

For $t \in \C^*$, the map $\phi\colon
\bigl(\Omega^\hdot(\A),\d+t\omega_\lambda\bigr) \to
\bigl(\Omega^\hdot(\A),\omega_\lambda+\frac{1}{t} \d\bigr)$ defined by 
$\phi(\eta) = \bigl(\frac{1}{t}\bigr)^q \eta$ for $\eta \in
\Omega^q(\A)$ 
is a cochain map, and is an isomorphism.  This fact, together with 
Proposition \ref{prop:mono}, implies that
$H^\hdot(\Omega^\hdot(\A)_m,\omega_\lambda+t'\d)=0$ 
for $m\neq 0$ and that 
$H^\hdot(\Omega^\hdot(\A)_0,\omega_\lambda+t'\d) \cong 
H^\hdot(A(\A),\omega_\lambda)$ for all but finitely many $t'$.

The Farber spectral sequence of the complex 
$\bigl(\Omega^\hdot(\A)_0,\omega_\lambda+t'd\bigr)$ has $E_1$-term 
$H^\hdot(\Omega^\hdot(\A)_0,\omega_\lambda)$, and abuts to 
$H^\hdot(\Omega^\hdot(\A)_0,\omega_\lambda+t'd) 
\cong H^\hdot(A(\A),\omega_\lambda)$ 
for generic $t'$.  Consequently, 
the assumption that $H^p(A(\A),\omega_\lambda)\neq 0$ implies that
$H^p(\Omega^\hdot(\A)_0,\omega_\lambda)\neq 0$ as well.  Hence, 
$H^p(\Omega^\hdot(\A),\omega_\lambda)\neq 0$.  Now use
Proposition~\ref{prop:tamecodim}: if $\A$ is free or $p\leq2$, 
the codimension of $\OSigma_\lambda$ is at most $p$.
\end{proof}

\section{Examples and Counterexamples} \label{sec:examples}


If $\Phi_\la$ is a master function, recall that $\LL_\la$ denotes the
corresponding complex, rank one
local system on the complement $M$ of the underlying arrangement $\A$. 
As noted in the Introduction, for sufficiently generic weights $\la$,
the inclusion of the Orlik-Solomon complex $(A(\A),\omega_\la)$ in the
twisted de Rham complex $(\Omega^\hdot(*\A),\d+\omega_\la)$ induces an
isomorphism 
$H^\hdot(A(\A),\omega_\la) \cong H^\hdot(M;\LL_\la)$.  See \cite{ESV,STV} for
conditions on $\la$ which insure that this isomorphism holds.

In light of this relationship between the Orlik-Solomon cohomology
$H^\hdot(A(\A),\omega_\la)$ and the local system cohomology
$H^\hdot(M;\LL_\la)$, one might expect a correspondence between the
non-vanishing of local system cohomology and the codimension of the
critical set of $\Phi_\la$, analogous to that established in Theorem
\ref{thm:res2crit}.  Such a correspondence does not hold, as the
following family of examples illustrate.

\begin{example} \label{ex:monomial deletions}
Let $r$ be a natural number, and $\alpha, \beta, \gamma$ complex numbers.
The master function
\[
\Phi = x_1^{r\alpha}
x_2^{r\beta}(x_1^r-x_2^r)^\gamma(x_1^r-x_3^r)^\beta(x_2^r-x_3^r)^\alpha
\]
determines a local system $\LL$ on the complement $M$ of the arrangement
$\A$ with defining polynomial $Q(\A)=x_1 x_2
(x_1^r-x_2^r)(x_1^r-x_3^r)(x_2^r-x_3^r)$.  Note that $\A$ has $3r+2$
hyperplanes, and let $\la\in\C^{3r+2}$ denote the collection of weights
corresponding to $\Phi$.  The one-form $\omega_\la=\d \log \Phi$ is given by
$\omega_\la=d_1 \d x_1+d_2 \d x_2+d_3 \d x_3$, where
\[
d_1=\frac{r\alpha}{x_1}+\frac{rx_1^{r-1}\gamma}{x_1^r-x_2^r}+\frac{rx_1^{r-1}
\beta}{x_1^r-x_3^r},\ 
d_2=\frac{r\beta}{x_2}+\frac{rx_2^{r-1}\gamma}{x_2^r-x_1^r}+\frac{rx_2^{r-1}
\alpha}{x_2^r-x_3^r},\ 
d_3=\frac{rx_3^{r-1}\alpha}{x_3^r-x_2^r}+\frac{rx_3^{r-1}\beta}{x_3^r-x_1^r},
\]
and the critical set $\Sigma_\la=V(\omega_\la)$ by $\Sigma_\la=V(d_1,d_2,d_3) 
\subseteq M$.

The arrangement $\A$ is supersolvable, hence free.  The module $\der(\A)$ has 
basis 
\[
\begin{aligned}
D_1&=x_1 \dd{}{x_1}+x_2 \dd{}{x_2}+x_3 \dd{}{x_3},\qquad
D_2=x_1^{r+1} \dd{}{x_1}+x_2^{r+1} \dd{}{x_3}+x_3^{r+1} \dd{}{x_3},\\
D_3&=x_1x_2(x_1x_2x_3)^{r-1}\left(x_1^{1-r} \dd{}{x_1}+x_2^{1-r} 
\dd{}{x_3}+x_3^{1-r} \dd{}{x_3}\right),
\end{aligned}
\]
see \cite[Prop. 6.85]{ot}. Consequently, the ideal $I_\la$ is generated by 
$d_i'=\langle D_i,\omega_\la\rangle$, $1\le i\le 3$, where
\[
d_1'=r(2\alpha+2\beta+\gamma),\ 
d_2'=r(\alpha+\beta+\gamma)(x_1^r+x_2^r)+r(\alpha+\beta)x_3^r,\ 
d_3'=r(\beta x_1^r+\alpha x_2^r)x_3^{r-1},
\]
and $\OSigma_\la=V(I_\la)=V(d_1',d_2',d_3') \subseteq\C^3$.  Observe that if 
$2\alpha+2\beta+\gamma \neq 0$, then $\OSigma_\la=\emptyset$ is empty, and 
hence $\Sigma_\la=\OSigma_\la\cap M=\emptyset$ is empty as well.

Let $q$ be a natural number with $1\le q \le r-1$, and assume that $\alpha,
\beta, \gamma$ satisfy $\alpha+\beta+\gamma \in \Z$ and $\gamma=-q/r$.  In 
this instance, 
it is known that the first local system cohomology
group is non-zero, $H^1(M;\LL_\la)\neq 0$, while the first Orlik-Solomon
cohomology group vanishes, $H^1(A(\A),\omega_\la)=0$, see
\cite{triples,Suciu}.  However, for such $\alpha,\beta, \gamma$, one has 
$2\alpha+2\beta+\gamma \neq 0$, so $\Sigma_\la=\emptyset$ and 
$\OSigma_\la=\emptyset$ as noted above.

Other choices of $\alpha,\beta, \gamma$ may be used to illustrate that the 
variety $\OSigma_\la$ is not, in general, the closure of $\Sigma_\la$, in 
contrast to the result of
Theorem~\ref{thm:closure} for the variety $\Sigma$.  This is the case, for
 example, if $\alpha+\beta=0$ and $\gamma=0$.  Here, 
$\OSigma_\la=V((x_1^r-x_2^r)x_3)$, while $\Sigma_\la=V(x_3)$.
\end{example}

The last example above may also be used to show that a converse of Theorem 
\ref{thm:res2crit} cannot hold.  That is, a master function with 
positive-dimensional critical set need not, in general, correspond to 
weights which are resonant in the corresponding dimension.

\begin{example} \label{ex:further deletion}
Let $r$ be a natural number, and $\alpha, \beta$ complex numbers.
The master function
\[
\Phi = x_1^{r\alpha}
x_2^{r\beta}(x_1^r-x_3^r)^\beta(x_2^r-x_3^r)^\alpha
\]
determines a local system $\LL$ on the complement $M$ of the arrangement
$\A$ with defining polynomial $Q(\A)=x_1 x_2
(x_1^r-x_3^r)(x_2^r-x_3^r)$.  Note that $\A$ has $2r+2$
hyperplanes, and let $\la\in\C^{2r+2}$ denote the collection of weights
corresponding to $\Phi$. The arrangement $\A\subset\C^3$ is not free, but is 
tame.

The one-form $\omega_\la=\d \log \Phi$ is given by
$\omega_\la=d_1 \d x_1+d_2 \d x_2+d_3 \d x_3$, where
\[
d_1=\frac{r\alpha}{x_1}+\frac{rx_1^{r-1}\beta}{x_1^r-x_3^r},\ 
d_2=\frac{r\beta}{x_2}+\frac{rx_2^{r-1}\alpha}{x_2^r-x_3^r},\ 
d_3=\frac{rx_3^{r-1}\alpha}{x_3^r-x_2^r}+\frac{rx_3^{r-1}\beta}{x_3^r-x_1^r},
\]
and the critical set $\Sigma_\la=V(\omega_\la)$ by $\Sigma_\la=V(d_1,d_2,d_3) 
\subseteq M$.
If $\alpha+\beta=0$, it is readily checked that $\Sigma_\la=V(x_3) \subset M$ 
is one-dimensional.
However, if $\alpha\neq 0$, then $H^1(A(\A),\omega_\la)=0$, and if the local 
system $\LL_\la$ corresponding to $\la$ is nontrivial, then $H^1(M;\LL_\la)=0$.
\end{example}

\begin{example}\label{ex:ER}
Consider the arrangement in $\P^3$ given by the nine linear forms
$x_1$, $x_2$, $x_3$, $x_i+x_4$ for $1\leq i\leq 3$, and $x_i+x_j+x_4$, 
for $1\leq i<j\leq 3$.  A computation with Macaulay~2~\cite{M2} shows
that $S/I$ is not Cohen-Macaulay: the projective dimension of $S/I$ is
$5$, while the codimension is $4$.  It follows that the arrangement is
not tame, which can also be verified directly.
Accordingly, the ideal $I$ has an embedded prime $(x_1,x_2,x_3,x_4)$, so 
we see that Corollary~\ref{cor:Iprime} 
requires the hypothesis that $\A$ is tame.

On the other hand, 
further calculation shows that the complex \eqref{eq:logcplx}
is exact for this arrangement, in contrast to Example~5.6 of \cite{ot95b}.
It would be interesting to know, then, if Theorem~\ref{th:two} holds 
without hypothesis.  For this example, the logarithmic comparison
isomorphism $H^\hdot(\Omega^\hdot(\A),\d)\cong A(\A)$ holds, since the rank
is $4$, by \cite[Corollary~6.3]{WiensYuz}.  However, we also do not
know if this isomorphism holds in general.
\end{example}
\begin{ack}
The second author would like to thank Mathias Schulze for pointing out
an error in the previous version of the proof of Theorem~\ref{thm:closure}.
\end{ack}


\newcommand{\arxiv}[1]
{\texttt{\href{http://arxiv.org/abs/#1}{arxiv:#1}}}

\renewcommand{\MR}[1]
{\href{http://www.ams.org/mathscinet-getitem?mr=#1}{MR#1}}

\end{document}